\documentclass[12pt]{amsart}
\usepackage{amssymb,amsmath,amsfonts,latexsym}
\usepackage{bm}
\usepackage[all,cmtip]{xy}
\usepackage{amscd}

\newcommand{\newsection}[1]{\setcounter{equation}{0} \section{#1}}
\newcommand{\bea}{\begin{eqnarray}}
\newcommand{\eea}{\end{eqnarray}}

\newcommand{\clb}{\mathcal{B}}

\newcommand{\cld}{\mathcal{D}}
\newcommand{\cle}{\mathcal{E}}
\newcommand{\clf}{\mathcal{F}}

\newcommand{\clh}{\mathcal{H}}
\newcommand{\clk}{\mathcal{K}}
\newcommand{\cll}{\mathcal{L}}
\newcommand{\clm}{\mathcal{M}}

\newcommand{\clq}{\mathcal{Q}}

\newcommand{\cls}{\mathcal{S}}

\newcommand{\clw}{\mathcal{W}}

\newcommand{\z}{\bm{z}}
\newcommand{\w}{\bm{w}}

\newcommand{\D}{\mathbb{D}}
\newcommand{\C}{\mathbb{C}}

\newcommand{\raro}{\rightarrow}

\def\textmatrix#1&#2\\#3&#4\\{\bigl({#1 \atop #3}\ {#2 \atop #4}\bigr)}
\def\dispmatrix#1&#2\\#3&#4\\{\left({#1 \atop #3}\ {#2 \atop #4}\right)}
\newcommand{\be}{\begin{equation}}
\newcommand{\ee}{\end{equation}}
\newcommand{\ben}{\begin{eqnarray*}}
	\newcommand{\een}{\end{eqnarray*}}

\newcommand{\NI}{\noindent}

\newcommand{\bi}{\begin{itemize}}
	\newcommand{\ei}{\end{itemize}}

\newtheorem{Theorem}{\sc Theorem}[section]

\newtheorem{Corollary}[Theorem]{\sc Corollary}

\theoremstyle{definition}

\newtheorem{Remark}[Theorem]{\sc Remark}

\theoremstyle{plain}

\newtheorem{thm}{Theorem}[section]

\newtheorem{lem}[thm]{Lemma}

\theoremstyle{definition}

\numberwithin{equation}{section}

\let\phi=\varphi

\begin{document}

	\title[Doubly commuting mixed invariant subspaces in the polydisc]{Doubly commuting mixed invariant subspaces in the polydisc}

   \author[Maji]{Amit Maji}
	\address{Indian Institute of Technology Roorkee, Department of Mathematics,
		Roorkee-247 667, Uttarakhand,  India}
	\email{amit.maji@ma.iitr.ac.in, amit.iitm07@gmail.com}
	

	\author[T R]{Sankar T R}
	\address{Indian Statistical Institute, Statistics and Mathematics Unit, 8th Mile, Mysore Road, Bangalore, 560059, India}
	\email{sankartr90@gmail.com}
	
	


\maketitle

\begin{abstract}
We obtain a complete characterization for {\it{doubly commuting mixed 
invariant}} subspaces of the Hardy space over the unit polydisc.
We say a closed subspace $\clq$ of $H^2(\D^n)$ is
{\it{mixed invariant}} if
$M_{z_{j}}(\clq) \subseteq \clq$ for 
$1 \leq j \leq k$ and $M_{z_{j}}^*(\clq) \subseteq \clq$, $k+1 \leq j \leq n$
for some integer $k \in \{1, 2, \ldots, n-1 \}$.
We prove that a {\it{mixed invariant}} subspace $\clq$ of $H^2(\D^n)$ 
is doubly commuting if and only if
\[
\clq = \Theta H^2(\D^k) \otimes \clq_{\theta_1} \otimes \cdots \otimes \clq_{\theta_{n-k}},
\]
where $\Theta \in H^{\infty}(\D^k)$ is some inner function and $\clq_{\theta_j}$ is either a Jordan block $H^2(\D)\ominus \theta_j H^2(\D)$
for some inner function $\theta_j$ or the Hardy space $H^2(\D)$. 
Furthermore, an explicit representation for the commutant of
an $n$-tuple of doubly commuting shifts as well as a 
representation for the commutant of a doubly commuting tuple of shifts and co-shifts are obtained. Finally, we discuss some concrete examples of {\it{mixed invariant}} subspaces.  
\end{abstract}


\newsection{Introduction}

The celebrated Beurling theorem (see \cite{B}) gives the structure of
shift invariant subspaces in the Hardy space $H^2(\D)$ over the unit disc: 
A non-zero closed subspace $\cls$ of
$H^2(\D)$ is invariant for $M_z$ if and only if there exists an
inner function $\theta \in H^\infty(\D)$ such that
\[
\cls = \theta H^2(\D).
\]
In particular, we can represent $\cls$ as
\[
\cls = \mathop{\oplus}_{m=0}^\infty z^m (\cls \ominus z \cls).
\]
A proper closed subspace $\clq$ of $H^2(\D)$ is said to be a Jordan block of 
$H^2(\D)$ if $\clq$ is $M_z^*$-invariant. Now $\clq$ is $M_z^*$-invariant if and only if $\clq^{\perp}$ is $M_z$-invariant. Thus by Beurling's theorem, there exists
an inner function $\theta \in H^{\infty}(\D)$ such that $\clq^{\perp} = \theta H^2(\D)$.
Therefore the Jordan block of $H^2(\D)$ is determined by  $\theta$ and is given by $\clq_{\theta} = H^2(\D)\ominus \theta H^2(\D)$.

Now we are going to define Jordan block of the Hardy space over the unit polydisc. 
It is well known that the $n$-tuple
of multiplication operators $(M_{z_1}, \ldots, M_{z_n})$ on the Hardy space 
$H^2(\D^n)$, $n>1$ over the unit polydisc is doubly commuting. A proper closed subspace $\clq $
of $ H^2(\D^n)$ is said to be a Jordan block of 
$H^2(\D^n)$, $n>1$ if $\clq$ is invariant under the adjoint of $(M_{z_1}, \ldots, M_{z_n})$
and the $n$-tuple $(P_{\clq}M_{z_1}|_{\clq}, \ldots, P_{\clq}M_{z_n}|_{\clq})$ is doubly commuting.
Firstly, Douglas and Yang (see \cite{DY-Quotient}, \cite{DY-Operator}) 
started the study of Jordan block 
of the Hardy space $H^2(\D^2)$ over the bidisc. 
After that, Izuchi, Nakazi and Seto \cite{INS}   
classified the Jordan block of $H^2(\D^2)$. Recently, J. Sarkar \cite{JAY-JORDAN}
has given a complete characterization of the Jordan block of $H^2(\D^n)$: 
A subspace $\clq $ of $H^2(\D^n)$ is a Jordan block if and only if
$\clq = \clq_{\theta_1} \otimes \cdots \otimes \clq_{\theta_n}$, where $\clq_{\theta_i}$
is a Jordan block of $H^2(\D)$.

The structure of joint invariant subspaces in the Hardy space $H^2(\D^2)$ over the bidisc
and in general the Hardy space over the polydisc $H^2(\D^n)$, $n > 1$ is very complex. That is because of the subtleties of theory of holomorphic functions in several variables as well as the difficulty associated with the structure of
$n$-tuples, $n >1$, of commuting isometries on Hilbert spaces.
There are several mathematicians like Agrawal,
Clark, and Douglas \cite{ACD}, Ahern and Clark \cite{AC}, Guo
\cite{KG1, KG}, Guo, Sun, Zheng and Zhong \cite{GZ},
Rudin \cite{Ru}, Izuchi \cite{I}, Mandrekar
\cite{M} etc. (also see \cite{III-Ranks, III-Blaschke, IN, CG} and the references therein) 
have studied the structure
of joint invariant subspaces. 
The four
operators $R_{z_1}, R_{z_2}$, $C_{z_1}$, and $C_{z_2}$ play an important role in finding the structure of (joint) invariant subspace $\cls$ of $H^2(\D^2)$,
where
$R_{z_1}= M_{{z_1}}|_{\cls}$ and $R_{z_2}= M_{{z_2}}|_{\cls}$, $C_{z_1} =P_{\cls^{\perp}}M_{{z_1}}|_{\cls^{\perp}}$,
and $C_{z_2} =P_{\cls^{\perp}}M_{{z_2}}|_{\cls^{\perp}}$.
For example,
Mandrekar \cite{M} showed that if the commutator $[R_{z_1}, R_{{z_2}}^{*}]=0$, 
then $\cls= \phi H^2(\D^2)$ for some inner function $\phi \in H^{\infty}(\D^2)$. 
Again if $[C_{z_1}, C_{{z_2}}^{*}]=0$
or $\mbox{rank} [C_{z_1}, C_{{z_2}}^{*}] =1$,
then the structure of $\cls$ is known (see \cite{INS}). 
But for $\mbox{rank} [C_{z_1}, C_{{z_2}}^{*}] >1$,
the structure of $\cls$ is not known.  
Motivated by the above fact, Izuchi et. al. (see \cite{IIN}) firstly introduced
the concept of {\it{mixed invariant}} subspaces of $H^2(\D^2)$:  
A closed subspace
$\cls$ of $H^2(\D^2)$ is said to be {\it{mixed invariant}} if
$M_{z_1}(\cls) \subseteq \cls$ and $M^{*}_{z_2}(\cls) \subseteq \cls$.
Clearly, if $\cls$ is
joint $(M_{z_1}, M_{z_2}^*)$ invariant, then $\cls^{\perp} = H^2(\D^2) \ominus \cls$ is again a joint
$(M_{z_1}^*, M_{z_2})$ invariant. 
Thus a natural question arises in the setting of the Hardy space over the unit polydisc:

\vspace{.1cm}
\textsf{What is the structure of  
{\textit{doubly commuting mixed invariant}} subspaces of $H^2(\D^n)$, $n>1$}?
\vspace{.1cm}

\noindent
More precisely, if a proper closed subspace $\clq$ of $H^2(\D^n)$ for $n>1$ is invariant
under \\
$(M_{z_1}, \ldots, M_{z_k}, M^*_{z_{k+1}},\ldots, M^*_{z_n})$ and the $n$-tuple
$(P_{\clq}M_{z_1}|_{\clq}, \ldots, P_{\clq}M_{z_n}|_{\clq})$ is doubly commuting, then
what is the explicit structure of $\clq$?

In this paper we will give a complete characterization of
{\textit{doubly commuting mixed invariant}} 
subspaces of $H^2(\D^n)$, $n>1$. 
We also find an explicit representation for the commutant of an $n$-tuple of 
doubly commuting shifts.
Moreover, a representation of a large class of bounded operators 
which intertwine with some doubly commuting shifts and co-shifts
on the Hardy space over the polydisc is obtained.

This paper is organized as follows: In Section 2 we give some 
basic definitions and results on the Hardy space over the
polydisc. In Section 3, we find an explicit representation for the commutant
of an $n$-tuple of doubly commuting shifts and also a 
representation for the commutant of a doubly commuting tuple of shifts and co-shifts 
on the Hardy space over the polydisc.
In Section 4, we completely classify {\textit{doubly commuting mixed invariant}} subspaces
of the Hardy spaces over the polydisc. Finally, we discuss some applications in Section 5.


\newsection{Preliminaries}
For $n \geq 1$, let $\D^n = \D \times \dots \times \D$ be the open unit polydisc in $\C^n$.
Throughout the paper we use the notation $\z$ for the
$n$-tuple $(z_1, \ldots, z_n)$ in $\mathbb{C}^n$. Also for any
multi-index $\bm{k} = (k_1, \ldots, k_n) \in \mathbb{Z}_+^n$ and $\z
\in \mathbb{C}^n$, we write $\z^{\bm{k}} = z_1^{k_1} \cdots
z_n^{k_n}$. Let $\clh$ and $\clk$ be two separable complex Hilbert spaces. 
The set of all bounded linear operators from $\clh$ to $\clk$ is denoted 
by $\clb(\clh, \clk)$. If $\clh = \clk$, then we shall denote it by $\clb(\clh)$
instead of $\clb(\clh, \clh)$.
A contraction $T$ on $\clh$ (that is, $\|Th \| \leq \|h \|$ for all $h \in \clh$)
is said to be a pure contraction if $T^{*m} \rightarrow 0$ as $m \rightarrow \infty$
in the strong operator topology. Also 
an isometry $V$ on $\clh$ (that is, $V^{*}V =I_{\clh}$ )
is said to be pure or shift if $V^{*m} \rightarrow 0$ as $m \rightarrow \infty$
in the strong operator topology (see \cite{H}, \cite{NF}).

The \textit{Hardy space}
$H^2(\mathbb{D}^n)$ over the polydisc $\mathbb{D}^n$ is the Hilbert space of all
holomorphic functions $f$ on $\mathbb{D}^n$ such that
\[
\|f\|_{H^2(\mathbb{D}^n)} = \left(\sup_{0\leq r< 1}
\int_{\mathbb{T}^n}|f(r e^{i \theta_1}, \ldots, r e^{i
\theta_n})|^2~d {\theta} \right)^{\frac{1}{2}}< \infty,
\]
where $d {\theta}$ is the normalized Lebesgue measure on the $n$-dimensional 
torus $\mathbb{T}^n$, the distinguished boundary of $\mathbb{D}^n$. 
It is a well known and important reproducing kernel Hilbert space
corresponding to the Szeg\"{o} kernel $\mathbb{S}_n$ on $\D^n$,
where
\[
\mathbb{S}_n(\z, \w) = \prod_{j=1}^n (1 - z_j \bar{w}_j)^{-1} \quad
\quad (\z, \w \in \D^n).
\]
Then one can see that
\[
\mathbb{S}_n^{-1}(\z, \w) = \mathop{\sum}_{0 \leq |\bm{k}| \leq n}
(-1)^{|\bm{k}|} \z^{\bm{k}} \bar{\w}^{\bm{k}},
\]
where $|\bm{k}| = \mathop{\sum}_{j=1}^n k_j$ and $0 \leq k_j \leq 1$
for all $j = 1, \ldots, n$.

We can extend this definition to a Hilbert space valued Hardy space also. Let $\cle$ be a Hilbert space, and the $\cle$-valued Hardy space over $\D^n$ be denoted by $H^2_{\cle}(\D^n)$. Then $H^2_{\cle}(\D^n)$ is
the $\cle$-valued reproducing kernel Hilbert space with the
$\clb(\cle)$-valued kernel function
\[
(\z, \w) \mapsto \mathbb{S}_n(\z, \w) I_{\cle} \quad \quad (\z, \w
\in \D^n).
\]
In the sequel, 
we shall identify the vector valued Hardy space $H^2_{\cle}(\D^n)$ with $H^2(\D^n)
\otimes \cle$ by the help of the canonical unitary $U :
H^2_{\cle}(\D^n) \raro H^2(\D^n) \otimes \cle$ defined by
\[
U (\z^{\bm{k}} \eta) = \z^{\bm{k}} \otimes \eta \quad \quad (\bm{k}
\in \mathbb{Z}^n_+, \eta \in \cle).
\]
One can easily observe that
the Hardy space $H^2(\D^n)$, $n>1$ over the unit polydisc can be identified 
via the unitary map $U : H^2(\D^n) \rightarrow H^2(\D) \otimes \cdots \otimes H^2(\D)$, 
where 
\[
{U}(z_1^{k_1} z_2^{k_2} \cdots z_{n}^{k_{n}} ) = z^{k_1}
\otimes z^{k_2} \otimes \cdots \otimes z^{k_n},
\]
with the $n$-fold Hilbert space tensor product of the Hardy space $H^2(\D)$ 
on the unit disc. 
Also we can realize $H^2(\D) \otimes \cdots \otimes H^2(\D)$ as the vector-valued
Hardy space $H^2_{{H^2(\D^{n-1})}}(\D)$ over the unit disc. 
This identification will be helpful to understand the structure of invariant subspaces
of $H^2(\D^n)$.

Let $(M_{z_1}, \ldots,
M_{z_n})$ denote the $n$-tuple of multiplication operators on
$H^2_{\cle}(\D^n)$ by the coordinate functions $\{z_j\}_{j=1}^n$,
that is,
\[
(M_{z_j} f)(\w) = w_j f(\w),
\]
for all $f \in H^2_{\cle}(\D^n)$, $\w \in \D^n$ and $j = 1, \ldots,
n$. It is easy to see that $(M_{z_1}, \ldots,
M_{z_n})$ is an $n$-tuple of \textit{doubly commuting} shifts on $H^2_{\cle}(\D^n)$. 
Evidently, the shift $M_{z_j}$ on
$H^2_{\cle}(\D^n)$ can be identified with $M_{z_j} \otimes I_{\cle}$
on $H^2(\D^n) \otimes \cle$ for $j = 1, \ldots,
n$. This canonical identification will be
used throughout the paper.
 We recall that a closed subspace $\cls$ of $ H^2_{\cle}(\D^n)$ 
is called an {invariant subspace} for
$(M_{z_1}, \ldots, M_{z_n})$ if
$M_{z_j} (\cls) \subseteq \cls$
for all $j = 1, \ldots, n$. 
Let $I$ be any nonempty and proper subset of $\{1, 2, \ldots, n \}$.
We say a closed subspace $\clq$ of $H^2(\D^n)$ is \textit{mixed invariant} if
$M_{z_{j}}(\clq) \subseteq \clq$ for 
$j \in I$, and $M_{z_{j}}^*(\clq) \subseteq \clq$ for $j \in I^{c}$. 
Without loss of generality,
we say that $\clq$ is {\it{mixed invariant}}
under the multiplication operators $(M_{z_1}, \ldots, M_{z_n})$ 
if $M_{z_{j}}(\clq) \subseteq \clq$ for 
$1\leq j \leq k$ and $M_{z_{j}}^*(\clq) \subseteq \clq$, $k+1 \leq j \leq n$
for some integer $k \in \{1, \ldots, n-1 \}$. In addition, $\clq$ is said to be
{\textit{doubly commuting mixed invariant}} if the $n$-tuple
$(P_{\clq}M_{z_1}|_{\clq}, \ldots, P_{\clq}M_{z_n}|_{\clq})$ is doubly commuting.
We denote $H^\infty_{\clb(\cle)}(\D)$ as the Banach
algebra of all $\clb(\cle)$-valued bounded analytic functions on the
open unit disc $\D$ (see \cite{NF}).


\newsection{commutant of a doubly commuting shifts}

In this section we obtain an explicit representation
of the commutant of an $n$-tuple of doubly commuting shifts.
We also find a representation of a large class of
bounded operators which intertwine with shifts and adjoint of 
shifts, that is, co-shifts. As a by-product we give some examples of
 {\it{mixed invariant}} subspaces 
in the Hardy space over the polydisc.

Let $m, n$ be positive integers. For the sake of simplicity we denote $H^2(\D^{m})$ by $H_{m}$. Here we use a similar technique used in our earlier paper (see \cite{AASS}).
We firstly denote $z_j$ as the coordinate function on $H^2(\D^n)$ for $j=1, \ldots, n$
and $w_i$ as the coordinate function on $H_m$ for $i=1, \ldots, m$.
 We then identify the Hardy space 
$H^2(\D^{n+m})$ as $H_{m}$-valued Hardy space $H^2_{H_{m}}(\D^n)$ over $\D^n$ 
via canonical unitary. With this identification the multiplication operators $(M_{z_1}, \ldots, M_{z_n}, \ldots, M_{z_{n+m}} )$ on $H^2(\D^{n+m})$ can be represented by $(M_{z_1}, \ldots, M_{z_n}, M_{\kappa_1}, \ldots, M_{\kappa_{m}} )$ 
on $H^2_{H{m}}(\D^n)$, where $\kappa_i$ is defined by 
\[
\kappa_i (\bm{z}) = M_{w_{i}} ,
\]
for all $\bm{z} \in \D^n$.
 It is evident that $\kappa_i
\in H^\infty_{\clb(H{_m})}(\D^n)$ is a constant function and the multiplication operator
$M_{\kappa_i}$ on $H^2_{H_{m}}(\D^n)$ is defined by
\[
M_{\kappa_i} f = \kappa_i f \quad \quad (f \in H^2_{H_{m}}(\D^n)),
\]
is a shift on $H^2_{H_{m}}(\D^n)$ for all $i = 1, \ldots, m$.

We now identify $H^2(\D^{n+m})$ with $H^2(\D^n) \otimes
H_m$ by the canonical unitary map 
$\hat{U} : H^2(\D^{n+m}) \raro H^2(\D^n) \otimes H_m$ defined by
\[
\hat{U}(z_1^{k_1} \cdots z_n^{k_n} \cdots z_{n+m}^{k_{n+m}} ) = 
(z_1^{k_1} \cdots z_n^{k_n})\otimes (w_{1}^{k_{n+1}} \cdots w_{m}^{k_{n+m}}),
\]
for all $k_1, \ldots, k_{n+m} \geq 0$. Then one can easily see that
\[
\hat{U} M_{z_j} = (M_{z_j} \otimes I_{H_m}) \hat{U} \qquad (j =1, \ldots, n)
\]
and
\[
\hat{U} M_{z_{n + i}} = (I_{H^2(\D^n)} \otimes K_i) \hat{U} \qquad (i =1, \ldots, m),
\]
where $K_i = M_{w_{i}}$.
Thus the $(n+m)$-tuples $(M_{z_1}, \ldots, 
M_{z_n}, \ldots, M_{z_{n+m}})$ on $H^2(\D^{n+m})$ and 
$(M_{z_1} \otimes I_{H_m}, \ldots, M_{z_n} \otimes I_{H_m}, I_{H^2(\D^n)} \otimes K_1, \ldots, I_{H^2(\D^n)}
\otimes K_m)$ on $H^2(\D^n) \otimes H_m$ are unitarily equivalent.

We now define a unitary map $\tilde{U} : H^2(\D^n) \otimes H_m \raro
H^2_{H_m}(\D^n)$ 
\[
\tilde{U}(z_1^{k_1} \cdots z_n^{k_n} \otimes f) = z_1^{k_1} \cdots z_n^{k_n} f,
\]
for all $k_j \geq 0$ and $f \in H_m$. Then
\[
\tilde{U} (M_{z_j} \otimes I_{H_m}) = M_{z_j}  \tilde{U} \qquad (j =1, \ldots, n).
\]

For $\kappa_i \in H^\infty_{\clb(H_m)}(\D^n)$,
and each operator $M_{\kappa_i}: H^2_{H_m}(\D^n) \raro H^2_{H_m}(\D^n)$
defined by
\[
(M_{\kappa_i} (z_1^{k_1} \cdots z_n^{k_n} f))(\bm{x}) = x_1^{k_1} \cdots x_n^{k_n} (K_i f),
\]
for all $k_j \geq 0$, $f \in H_m$ and $\bm{x} = (x_1, \ldots, x_n) \in \D^n$, is a shift on
$H^2_{H_m}(\D^n)$. It also follows that
\[
\tilde{U} (I_{H^2(\D^n)} \otimes K_i) = M_{\kappa_i} \tilde{U}.
\]
Therefore by setting
$U = \tilde{U} \hat{U}$,
we see that $U : H^2(\D^{n+m}) \raro H^2_{H_m}(\D^n)$ is
a unitary operator satisfying
\[
U M_{z_j} = M_{z_j} U,
\]
and
\[
U M_{z_{n + i}} = M_{\kappa_i} {U},
\]
for all $j = 1, \ldots, n$ and $i = 1, \ldots, m$. 

To summarize this, we have the following:

\begin{thm} \label{thm-1}
Let $m, n$ be positive integers and $H^2(\D^{n+m})$ be the Hardy space over the unit polydisc $\D^{n+m}$.
Then the $(n+m)$-tuples of multiplication operators $(M_{z_1}, \ldots, M_{z_n}, \ldots,
M_{z_{n+m}})$ on $H^2(\D^{n+m})$ and 
$(M_{z_1}, \ldots, M_{z_n}, M_{\kappa_1},
\ldots, M_{\kappa_m})$ on $H^2_{H_m}(\D^n)$ are unitarily
equivalent, where $\kappa_i \in H^\infty_{\clb(H_m)}(\D^n)$ is the
constant function defined by
\[
\kappa_i(\bm{z}) = M_{w_{i}} \in \clb(H_m),
\]
for all $\bm{z} \in \D^n$ and $i = 1, \ldots, m$.
\end{thm}

We now return to the representation of the commutant for a tuple of doubly commuting shifts.
Let $n$ be a fixed positive integer. 
Recall that for $\bm{z}= (z_1, \ldots, z_n) \in \mathbb{C}^n$ and for any
multi-index $\bm{k} = (k_1, \ldots, k_n) \in \mathbb{Z}_+^n$, 
we write $\z^{\bm{k}} = z_1^{k_1} \cdots z_n^{k_n}$. Also for an $n$-tuple of 
commuting shifts $V=(V_1, \ldots, V_n)$,  $V^{\bm{k}}$ denotes $V_1^{k_1} \ldots V_n^{k_n}$.
It is also well known (see \cite{DAN-P}) that
if $V=(V_1, \ldots, V_n)$ is an $n$-tuple of doubly
commuting shifts on a Hilbert space $\clh$, then
\[
\clh = \displaystyle{\oplus_{\bm{k} \in \mathbb{Z}^n_{+}} V^{\bm{k}}\clw},
\]
where $\clw = \displaystyle\cap_{j=1}^n Ker(V_j^{*})$.
We now define a map $\Pi: \clh \rightarrow H^2_{\clw}(\D^n)$ by
\begin{align}\label{unitary-dc}
\Pi(V_1^{k_1} \cdots  V_n^{k_n} \eta) = z_1^{k_1} \cdots z_n^{k_n} \eta \qquad (k_j \geq 0, \eta \in \clw).
\end{align}
Then $\Pi$ is clearly a unitary. Moreover,
\[
\Pi V_j = M_{z_j}\Pi \quad \mbox{for}~j=1, \ldots, n.
\]
It is well known that if a bounded operator $X \in \clb(H^2_{\cle}(\D^n))$ commutes with the 
multiplication operators $(M_{z_1}, \ldots, M_{z_n})$ on $H^2_{\cle}(\D^n)$, i.e.,
$XM_{z_j} = M_{z_j}X$ for all $1 \leq j \leq n$, then
$X = M_{\Theta}$ for some $\Theta \in H^{\infty}_{\clb(\cle)}(\D^n)$ (see \cite{BLTT} ). 
These are called analytic Toeplitz operators which is an important class of operators.
Also the commutant of $V=(V_1, \ldots, V_n)$ on $\clh$ can be represented as
the commutant of $(M_{z_1}, \ldots, M_{z_n})$ on $H^2_{\clw}(\D^n)$. 
Hence the commutant of $V=(V_1, \ldots, V_n)$ on $\clh$ can be written
as the form $M_{\Phi}$ on $H^2_{\clw}(\D^n)$, where 
$\Phi \in H^\infty_{\clb(\clw)}(\D^n)$. The question is 
what is the explicit representation of $\Phi$? 
Here we have the following result.

\begin{lem}\label{thm-commutator}
Let $\clh$ be a Hilbert space and $V=(V_1, \ldots, V_n)$ be an $n$-tuple of doubly
commuting shifts on $\clh$. Let $T \in \clb( \clh)$ and let $\Pi$ be the above unitary
defined in (\ref{unitary-dc}).
Then
\[
TV_j = V_jT \quad \mbox{~for all ~}j=1, \ldots, n
\]
if and only if $\Pi T \Pi^{*} =M_{\Phi}$ for some 
$\Phi \in H^\infty_{\clb(\clw)}(\D^n)$ and
\[
\Phi(\bm{z})= P_{\clw}(I_{\clh} -z_1 V_1^{*})^{-1} \ldots (I_{\clh} -z_n V_n^{*})^{-1} T\mid_{\clw}, 
\]
where $\bm{z}=(z_1, \ldots, z_n) \in \D^n$, $\clw = \cap_{j=1}^n Ker(V_j^{*})$.
\end{lem}

\begin{proof}
Let $h \in \clh$. Then $h$ can be written as 
\[
h =\displaystyle{\sum_{\bm{k} } V^{\bm{k}} \eta_{\bm{k}}}, 
\]
for some $\eta_{{\bm{k}}}\in \clw$, ${\bm{k}} =(k_1, \ldots, k_n) \in \mathbb{Z}^n_{+}$ 
(as $\clh = \displaystyle{\oplus_{\bm{k} \in \mathbb{Z}^n_{+}} V^{\bm{k}}\clw}$). 
Applying $P_{\clw} V^{* {\bm{m}} }$ to both sides and since
 $\clw =  \cap_{j=1}^n Ker(V_j^{*})$, we have $\eta_{{\bm{m}}} =
P_{\clw} V^{*{\bm{m}}} h$ for all ${\bm{m}} =(m_1, \ldots, m_n) \in \mathbb{Z}^n_{+}$. 
Therefore, for any $h \in \clh$,
\begin{equation}\label{eq-F}
h = \displaystyle{\sum_{\bm{k} } V^{\bm{k}} P_{\clw} V^{ *\bm{k}} h}.
\end{equation}
Now let $TV_j = V_jT$ {for all} $j=1, \ldots, n$. Then there exists a bounded analytic function
$\Phi \in H^{\infty}_{\clb(\clw)}(\D^n)$ such that $\Pi T \Pi^*
= M_{\Phi}$. For each $\bm{w} =(w_1, \ldots, w_n) \in \mathbb{D}^n$ and $\eta \in \clw$ we
have
\[
\begin{split}
\Phi(\bm{w}) \eta & = (M_{\Phi} \eta)(\bm{w})
\\
& = (\Pi T \Pi^* \eta) (\bm{w})
\\
& = (\Pi T \eta)(\bm{w}),
\end{split}
\]
as $\Pi^* \eta = \eta$. Using (\ref{eq-F}), we have
\[
T \eta = \displaystyle{\sum_{\bm{k} } V^{\bm{k}} P_{\clw} V^{ *\bm{k}} T \eta} ,
\]
it follows that
\[
\begin{split}
\Phi(\bm{w}) \eta & = (\Pi \left[\displaystyle{\sum_{\bm{k} } V^{\bm{k}} P_{\clw} V^{ *\bm{k}} T \eta} \right])(\bm{w})
\\
& = \left[\displaystyle{\sum_{\bm{k} } M_{\bm{z}}^{\bm{k}} P_{\clw} V^{ *\bm{k}} T \eta} \right](\bm{w})
\\
& = \left[\displaystyle{\sum_{\bm{k} } {\bm{w}}^{\bm{k}} P_{\clw} V^{ *\bm{k}} T \eta} \right]
\\
& = P_{\clw} \left[\displaystyle{\sum_{\bm{k} } {\bm{w}}^{\bm{k}}  V^{ *\bm{k}} T \eta} \right]
\\
& = P_{\clw} (I_{\clh} - w_1 V_1^*)^{-1} \cdots (I_{\clh} - w_n V_n^*)^{-1} T \eta.
\end{split}
\]
Therefore
\[
\Phi(\bm{z}) = P_{\clw} (I_{\clh} - z_1 V_1^*)^{-1} \cdots (I_{\clh} - z_n V_n^*)^{-1} T|_{\clw} \quad \quad
(\bm{z} \in \D^n),
\]
as required. We omit the sufficient part as it holds trivially. 

This completes the proof.
\end{proof}

\begin{Remark}\label{Remark-commutant-1}
In the proof we have used the standard projection
formula 
$I_{\clh} = \displaystyle{\mbox{SOT}-{\sum_{\bm{k} } V^{\bm{k}} P_{\clw} V^{ *\bm{k}}}}$. 
Note that the function $\Phi$
defined in Lemma \ref{thm-commutator} is well defined and a $\clb(\clw)$-valued
holomorphic function in the unit polydisc $\D^n$ as $\|w_j V_j^*\| = |w_j| \|V_j\| < 1$
for all $w_j \in \D$. Now if $T$ commutes with a shift $V_1$, then from the 
above result, we have the representation
\[
\Phi(z) = P_{\clw} (I_{\clh} - z V_1^*)^{-1} T|_{\clw} \quad \quad
(z \in \D, \clw = Ker(V_1^{*}))
\]
which was obtained in our previous paper \cite{MSS}.
\end{Remark}

We are now in a position to find the explicit representation 
of the commutant of a tuple of doubly commuting shifts and co-shifts.

\begin{thm}
Let $T$ be a bounded linear operator on $H^2_{H_m}(\D^n)$. Then $T$ satisfies
\[
TM_{z_j}=M_{z_j}T \quad  \mbox{for~} j=1, \ldots, n 
\]
and
\[
TM_{\kappa_i}^{*} = M_{\kappa_i}^{*}T \quad  \mbox{for ~} i=1, \ldots, m
\]
if and only if $T =M_{\Phi}$, where $\Phi \in H^\infty_{\clb(H_m )}(\D^n)$ and
\[
\Phi(\bm{z})= \sum_{\bm{k} \in \mathbb{Z}_{+}^n}M_{\Theta_{\bm{k}}}^{*} \bm{z}^{\bm{k}},
\]
and 
\[
 \Theta_{\bm{k}}(\bm{w})= \displaystyle \sum_{\bm{l} \in \mathbb{Z}_{+}^m}{\langle M^{\bm{k}}_{\bm{z}} \bm{1}, T \bm{w}^{\bm{l}} \rangle_{H^2_{H_m}(\D^n)} \bm{w}^{\bm{l}}}
\] 
for all $\bm{z} =(z_1, \ldots, z_n)\in \D^n$, $\bm{w} =(w_1, \ldots, w_m)\in \D^m$.
\end{thm}

\begin{proof}
Suppose that $T$ is a bounded linear operator on $H^2_{H_m}(\D^n)$ satisfying
\[
TM_{z_j}=M_{z_j}T \quad  \mbox{for~} j=1, \ldots, n 
\]
and
\[
TM_{\kappa_i}^{*} = M_{\kappa_i}^{*}T \quad  \mbox{for~} i=1, \ldots, m.
\]
From the first identity, we have $T = M_{\Phi}$, where $\Phi \in H^\infty_{\clb(H_m)}(\D^n)$.
Let
\[
\Phi( \bm{z})= \sum_{\bm{k}} T_{\bm{k}} {\bm{z}}^{\bm{k}} \quad ({\bm{z}}\in \D^n, ~ T_{\bm{k}} \in \clb(H_m)).
\]
Now $H^2_{H_m}(\D^n) = \overline{\mbox{span}} \{\bm{z}^{\bm{k}}\bm{w}^{\bm{l}}: \bm{k} \in \mathbb{Z}_{+}^n, ~\bm{l} \in \mathbb{Z}_{+}^m  \}$. 
Since $T$ satisfies
\[
TM_{\kappa_i}^{*} = M_{\kappa_i}^{*}T \quad  \mbox{for~} i=1, \ldots, m,
\]
we have
\begin{align*}
TM_{\kappa_i}^{*}( \bm{w}^{\bm{r}} )=TM_{\kappa_i}^{*}(\bm{w}^{\bm{r}} \bm{z}^0) & =  (\sum_{\bm{k}} T_{\bm{k}} {\bm{z}}^{\bm{k}}) M^{*}_{w_i} (\bm{w}^{\bm{r}}) \\ 
& = \sum_{\bm{k}} (T_{\bm{k}} M_{w_i}^{*}( \bm{w}^{\bm{r}} )){\bm{z}}^{\bm{k}} 
\end{align*}
and
\begin{align*}
M_{\kappa_i}^{*}T( \bm{w}^{\bm{r}} )=  M_{\kappa_i}^{*} (\sum_{\bm{k}} T_{\bm{k}} {\bm{z}}^{\bm{k}}) (\bm{w}^{\bm{r}}) = \sum_{\bm{k}} (M_{w_i}^{*}T_{\bm{k}} ( \bm{w}^{\bm{r}} )){\bm{z}}^{\bm{k}}. \\
\end{align*}
From the above we have,
\[
T_{\bm{k}} M_{w_i}^{*}( \bm{w}^{\bm{r}} ) = M_{w_i}^{*}T_{\bm{k}} ( \bm{w}^{\bm{r}} )
\]
for all $\bm{w} =(w_1, \ldots, w_m) \in \D^m$ and $\bm{k} \in \mathbb{Z}_{+}^n $. Therefore
\[
T_{\bm{k}} M_{w_i}^{*} = M_{w_i}^{*} T_{\bm{k}}  ~~~\mbox{ i.e.,}~~~M_{w_i}T_{\bm{k}}^{*} = T_{\bm{k}}^{*} M_{w_i}
\]
for all $i=1, \ldots, m$. Hence $T_{\bm{k}}^{*} = M_{\Theta_{\bm{k}}} $
for some $\Theta_{\bm k} \in H^\infty(\D^m)$ for all ${\bm k}$. Therefore
\[
\Phi( \bm {z})= \sum_{\bm{k}} M^{*}_{\Theta_{\bm {k}}} {\bm{z}^{\bm{k}}} \quad (\bm{z} \in \D^n).
\]

Now we will find $\Theta_{\bm {k}} \in H^\infty(\D^m)$ explicitly:

We know in view of Lemma \ref{thm-commutator} that
\begin{align*}
\Phi(\bm {z}) & =  \displaystyle{\sum_{\bm{k} } (P_{H_m}  M_{\bm {z}}^{ *\bm{k}} T\mid_{H_m }) {\bm{z}}^{\bm{k}} }.
\end{align*}
Hence
\[
M^{*}_{\Theta_{\bm {k}}}= P_{H_m}  M_{\bm {z}}^{ *\bm{k}} T\mid_{H_m },
\]
for all $\bm {k}$. 
We shall now calculate each $\Theta_{\bm{k}} \in H^\infty(\D^{m})$. 
Suppose $\Theta_{\bm{k}}( \bm{w}) = \sum_{\bm{l}} a_{\bm{l}}\bm{w}^{\bm{l}}$.
Then
\begin{align*}
\langle (\sum_{\bm{l}} a_{\bm{l}}\bm{w}^{\bm{l}})\bm{1},  \bm{w}^{\bm r} \rangle_{H^2(\D^m)} 
& = \langle M_{\Theta_{\bm{k}}} \bm{1},  \bm{w}^{\bm r} \rangle_{H^2(\D^m)} \\
& = \langle \bm{1}, M_{\Theta_{\bm{k}}}^{*} \bm{w}^{\bm r} \rangle_{H^2(\D^m)} \\
& = \langle \bm{1}, P_{H_m}  M_{\bm {z}}^{ *\bm{k}} T \bm{w}^{\bm r} \rangle_{H^2(\D^m)} \\
& = \langle  M_{\bm {z}}^{ \bm{k}} \bm{1},  T \bm{w}^{\bm r} \rangle_{H^2_{H_m}(\D^n)}.
\end{align*}
Since $\langle \bm {w}^{\bm {l}}, \bm {w}^{\bm {r}} \rangle =0$ for $ \bm{l} \neq \bm{r}$ 
and $\bm{w}^{0}=\bm{1}$, a constant function on $H^2(\D^m)$, from the above we have
$a_{\bm r} = \langle  M_{\bm {z}}^{ \bm{k}} \bm{1},  T \bm{w}^{\bm r} \rangle_{H^2_{H_m}(\D^n)}$ 
for all $\bm r \in \mathbb{Z}_{+}^m$. Thus
\[
\Phi( \bm {z})= \sum_{\bm{k}} M^{*}_{\Theta_{\bm {k}}} {\bm{z}^{\bm{k}}} \quad (\bm{z} \in \D^n).
\]
where
\[
\Theta_{\bm {k}} (\bm{w}) = \displaystyle \sum_{\bm{l}} \langle  M_{\bm {z}}^{ \bm{k}} \bm{1},  T \bm{w}^{\bm l} \rangle_{H^2_{H_m}(\D^n)}\bm{w}^{\bm{l}}.
\]
The converse follows easily. This finishes the proof.
\end{proof}

We know that if $T \in \clb( H^2(\D))$ such that $TM_z = M_zT$, then
$\overline{ran}(T)$ is an $M_z$-invariant subspace of $H^2(\D)$.
In other words, for any $\phi \in H^{\infty}(\D)$, $\overline{ran}(M_{\phi})$
is an $M_z$-invariant subspace of $H^2(\D)$. Similarly 
let $T \in \clb( H^2(\D^{n+m}))$.
Suppose that 
\[
TM_{z_j}=M_{z_j} T \quad  \mbox{and} \quad TM_{z_{n+i}}^{*} = M_{z_{n+i}}^{*}T,
\]
for $j =1, \ldots, n$ and $i =1, \ldots, m$. Then
$\overline{ran}(T)$ is invariant under 
$(M_{z_1},\ldots,  M_{z_n}, M_{z_{n+1}}^*, \ldots, M_{z_{n+m}}^*)$
on $H^2(\D^{n+m})$. Thus from the above theorem,
we can readily find some examples of {\it{mixed invariant}} subspaces.

\begin{Corollary}
Let $T \in \clb( H^2_{H_m}(\D^n))$. Then $\overline{ran}(T)$  
is a \it{mixed invariant} subspace under
$(M_{z_1},\ldots,  M_{z_n}, M_{\kappa_{1}}^*, \ldots, M_{\kappa_{m}}^*) $
if 
$T =M_{\Phi}$, where $\Phi \in H^\infty_{\clb(H_m)}(\D^n)$ is given by
\[
\Phi(\bm{z})= \sum_{\bm{k} \in \mathbb{Z}_{+}^n}M_{\Theta_{\bm{k}}}^{*} \bm{z}^{\bm{k}},
\]
and
\[
\Theta_{\bm{k}}(\bm{w})= \displaystyle \sum_{\bm{l} \in \mathbb{Z}_{+}^m}{\langle
M^{\bm{k}}_{\bm{z}} \bm{1}, T \bm{w}^{\bm{l}} \rangle_{H^2_{H_m}(\D^n)} \bm{w}^{\bm{l}}}
\]
for all $\bm{z} =(z_1, \ldots, z_n)\in \D^n$ and $\bm{w} =(w_1, \ldots, w_m)\in \D^m$.
\end{Corollary}

\section{Doubly commuting mixed invariant subspaces}
In this section we obtain an explicit representation of 
{\it{doubly commuting mixed invariant subspaces}}. 
We start with a known result (see {(\cite{AASS})} for more details).

\begin{lem} \label{Mul-lemma1}
	Let $\clf$ and $\cle$ be Hilbert spaces and
	$M_{\Phi}: H^2_{\clf}(\D^n) \rightarrow H^2_{\cle}(\D^n) $ be an isometric multiplier. 
	Then $\dim(\clf) \leq \dim(\cle)$. 
\end{lem}

The following result is required to prove our main results.
The idea of the proof comes from a paper of Douglas and Foias \cite{DF-UNIQUE}.
This result may be known to the experts but here our proof is 
direct and it is useful as an independent result.

\begin{lem}\label{lemma-extension1}
Let $\clq \subseteq H^2_{\cle}(\D^n)$ be an $M_{z_j}^{*}$-invariant subspace 
for all $j=1, 2, \cdots,n$ and $T: \clq \rightarrow H^2(\D^n)$
be an isometry such that $TM_{z_j}^{*}|_{\clq} = M_{z_j}^{*}T$ for all $j=1, 2, \cdots,n$. 
Then there exists a one dimensional Hilbert space $\tilde{\cle}$, a subspace of $\cle$
such that $\clq \subseteq H^2_{\tilde{\cle}}(\D^n)$ and a constant unitary multiplier
$M_{\Theta}: H^2_{\tilde{\cle}}(\D^n) \rightarrow  H^2(\D^n)$ 
such that $T = M_{\Theta}|_{\clq}$. 
\end{lem}

\begin{proof}
	Since $\clq $ is an $M_{z_j}^{*}$-invariant subspace for all $j=1, 2, \cdots,n$, the space
	\[
	\clm = \overline{span}\{ \bm{z}^{\bm{k}} \clq: {\bm{k}} \in \mathbb{Z}_{+}^n \} 
	\]
	is $M_{z_j}$-reducing subspace of $H^2_{\cle}(\D^n)$ for all $j=1, 2, \cdots,n$. We know that
	any $(M_{z_1},M_{z_2}, \ldots, M_{z_n})$ reducing subspace of $H^2_{\cle}(\D^n)$ is of the form $H^2_{\tilde{\cle}}(\D^n)$ 
	for some closed subspace ${\tilde{\cle}}$ of $\cle$. Thus $\clm = H^2_{\tilde{\cle}}(\D^n)$.
	Firstly, our aim is to extend the operator $T: \clq \rightarrow H^2(\D^n)$ to 
	$\tilde{T}: H^2_{\tilde{\cle}}(\D^n) \rightarrow H^2(\D^n)$	such that $\tilde{T}$ is an isometry
	and $M_{z_j}^{*} \tilde{T} = \tilde{T} M_{z_j}^{*}$ for all $j=1, 2, \cdots,n$.

	It is easy to see that the set $\{ \bm{z}^{\bm{k}} \eta: {\bm{k}} \in \mathbb{Z}_{+}^n, \eta \in \clq \} $ is a total set in $H^2_{\tilde{\cle}}(\D^n)$. 
	Consider
	\[
\cll = \mbox{span}\{ \bm{z}^{\bm{k}} \eta: {\bm{k}} \in \mathbb{Z}_{+}^n, \eta \in \clq \}.
	\]
	Let $F$ be a finite set in $\mathbb{Z}^n_{+}$. We define $\tilde{T} $ on $\cll$ as  
\[	
	\tilde{T} ( \sum_{\bm{k} \in F} \alpha_{\bm{k}} \bm{z}^{\bm{k}} \eta_{\bm{k}} ) 
	= \sum_{\bm{k} \in F} \alpha_{\bm{k}} \bm{z}^{\bm{k}} T \eta_{\bm{k}}, 
\]
where $\alpha_{\bm{k}} \in \mathbb{C}$ and $\eta_{\bm{k}} \in \clq$.  
	We shall firstly show that the map $\tilde{T}$ is well defined on $\cll$. To do that we consider
	the following. 
	
	Let $F_1, F_2$ be two finite sets in $\mathbb{Z}^n_{+}$, and 
	$\eta_{\bm{k}}, \zeta_{\bm{m}} \in \clq$. Then
\begin{align}\label{mainlemma_equa1}
	 < \sum_{{\bm{k}} \in F_1} \alpha_{\bm{k}} \bm{z}^{\bm{k}} T\eta_{\bm{k}}, \sum_{{\bm{m}} \in F_2} \beta_{\bm{m}} \bm{z}^{\bm{m}} T \zeta_{\bm{m}} > 	
	& = \sum_{{\bm{k}} \in F_1} \sum_{{\bm{m}} \in F_2} \alpha_{\bm{k}} \bar{ \beta_{\bm{m}}}  <  \bm{z}^{\bm{k}} T\eta_{\bm{k}}, \bm{z}^{\bm{m}} T \zeta_{\bm{m}} >. 
\end{align}	
Now for each fixed ${\bm{k}} \in F_1$ and ${\bm{m}} \in F_2$,
let ${\bm{k}}=(k_1, \ldots, k_n)$ and ${\bm{m}}=(m_1, \ldots, m_n)$. 
Consider $I = \{ j \in \{ 1, \cdots, n\}: k_j \geq m_j \}$.
Using the fact that $(M_{z_1}, \ldots,  M_{z_n})$ is a
tuple of doubly commuting isometries and $T$ is an isometry, we obtain
\begin{align*}
<  \bm{z}^{\bm{k}} T\eta_{\bm{k}}, \bm{z}^{\bm{m}} T \zeta_{\bm{m}} > 
&  = <\prod_{j=1}^n M_{z_j}^{k_j} T\eta_{\bm{k}}, \prod_{j=1}^n M_{z_j}^{m_j}T \zeta_{\bm{m}} > \\
& = <\prod_{j \in I} M_{z_j}^{k_j-m_j} T\eta_{\bm{k}}, \prod_{j \in I^c} M_{z_j}^{m_j-k_j} T \zeta_{\bm{m}}> \\
&= <\prod_{j \in I^c} M_{z_j}^{*(m_j-k_j)} T\eta_{\bm{k}}, \prod_{j \in I} M_{z_j}^{*(k_j-m_j)} T \zeta_{\bm{m}}> \\
& = <T\prod_{j \in I^c} M_{z_j}^{*(m_j-k_j)} \eta_{\bm{k}}, T\prod_{j \in I} M_{z_j}^{*(k_j-m_j)} \zeta_{\bm{m}}> \\
& = <\prod_{j \in I^c} M_{z_j}^{*(m_j-k_j)} \eta_{\bm{k}}, \prod_{j \in I} M_{z_j}^{*(k_j-m_j)} \zeta_{\bm{m}}> \\
& = < \prod_{j \in I} M_{z_j}^{(k_j-m_j)} \eta_{\bm{k}}, \prod_{j \in I^c} M_{z_j}^{(m_j-k_j)}\zeta_{\bm{m}}> 
\\
& = < \prod_{j \in I} {z_j}^{m_j} \prod_{j \in I^c} {z_j}^{k_j}(\prod_{j \in I} {z_j}^{(k_j -m_j)})\eta_{\bm{k}}, \prod_{j \in I} {z_j}^{m_j} \prod_{j \in I^c} {z_j}^{k_j}(\prod_{j \in I} {z_j}^{(m_j -k_j)}) \zeta_{\bm{m}}>\\
&= < \bm{z}^{\bm{k}} \eta_{\bm{k}}, \bm{z}^{\bm{m}} \zeta_{\bm{m}} >.
	\end{align*}
Thus from the above equation (\ref{mainlemma_equa1}), we have
\begin{align*}
	 < \sum_{{\bm{k}} \in F_1} \alpha_{\bm{k}} \bm{z}^{\bm{k}} T\eta_{\bm{k}}, \sum_{{\bm{m}} \in F_2} \beta_{\bm{m}} \bm{z}^{\bm{m}} T \zeta_{\bm{m}} > 
&=  < \sum_{{\bm{k}} \in F_1} \alpha_{\bm{k}}  \bm{z}^{\bm{k}} \eta_{\bm{k}},  \sum_{{\bm{m}} \in F_2}  { \beta_{\bm{m}}} \bm{z}^{\bm{m}} \zeta_{\bm{m}} >.
	\end{align*}
In particular for any finite set $F$ in $\mathbb{Z}^n_{+}$, we have
\begin{align}\label{mainlemma_equa2}
	\| \sum_{{\bm{k}} \in F} \alpha_{\bm{k}} \bm{z}^{\bm{k}} T \eta_{\bm{k}} \|^2 
	& = \| \sum_{{\bm{k}} \in F} \alpha_{\bm{k}} \bm{z}^{\bm{k}} \eta_{\bm{k}} \|^2. 
\end{align}
Suppose that
\[
\sum_{{\bm{k}} \in F_1} \alpha_{\bm{k}} \bm{z}^{\bm{k}} \eta_{\bm{k}} =  \sum_{{\bm{m}} \in F_2} \beta_{\bm{m}} \bm{z}^{\bm{m}} \zeta_{\bm{m}}. 
\]	
Then from the above equality (\ref{mainlemma_equa2}), it is easy to see that
\[
	\|  \sum_{{\bm{k}} \in F_1} \alpha_{\bm{k}} \bm{z}^{\bm{k}} T \eta_{\bm{k}} -  \sum_{{\bm{m}} \in F_2} \beta_{\bm{m}} \bm{z}^{\bm{m}} T \zeta_{\bm{m}} \|^2 = \| \sum_{{\bm{k}} \in F_1} \alpha_{\bm{k}} \bm{z}^{\bm{k}} \eta_{\bm{k}} -  \sum_{{\bm{m}} \in F_2} \beta_{\bm{m}} \bm{z}^{\bm{m}} \zeta_{\bm{m}} \|^2 =0. 
\]	
Therefore
\[
\sum_{{\bm{k}} \in F_1} \alpha_{\bm{k}} \bm{z}^{\bm{k}} T \eta_{\bm{k}} =  \sum_{{\bm{m}} \in F_2} \beta_{\bm{m}} \bm{z}^{\bm{m}} T \zeta_{\bm{m}}.
\]
Hence from the definition of $\tilde{T}$, we have	
\[
\tilde{T} (\sum_{{\bm{k}} \in F_1} \alpha_{\bm{k}} \bm{z}^{\bm{k}} \eta_{\bm{k}} )  
= \tilde{T} (\sum_{{\bm{m}} \in F_2} \beta_{\bm{m}} \bm{z}^{\bm{m}} \zeta_{\bm{m}}). 
\]	
This proves that $\tilde{T}$ is well defined on $\cll$. From equation (\ref{mainlemma_equa2}),
we can conclude that $\tilde{T}$ is
a bounded as well as norm preserving linear operator on $\cll$.

	Now ${\bm{k}} =(k_1, \ldots, k_n) \in \mathbb{Z}_{+}^n$, 
	$\eta \in \clq$, and for each $j=1, 2, \cdots, n$,  we have
	\begin{align*}
	\tilde{T}M_{z_j}( \bm{z}^{\bm{k}} \eta) = \tilde{T}(z_1^{k_1}\cdots z_j^{k_j+1} \cdots z_n^{k_n}  \eta) = (z_1^{k_1}\cdots z_j^{k_j+1} \cdots z_n^{k_n}) T\eta =  z_j\bm{z}^{\bm{k}} T\eta = M_{z_j}\tilde{T}(\bm{z}^{\bm{k}} \eta).
	\end{align*}
	Also if $k_j \geq 1$, then
	\begin{align*}
	\tilde{T}M_{z_j}^{*}( \bm{z}^{\bm{k}} \eta) = \tilde{T}(z_1^{k_1} \cdots z_j^{k_j-1} \cdots z_n^{k_n}  \eta) = (z_1^{k_1}\cdots z_j^{k_j-1} \cdots z_n^{k_n}) T\eta =  M_{z_j}^{*}(\bm{z}^{\bm{k}} T\eta) = M_{z_j}^{*} \tilde{T}(\bm{z}^{\bm{k}} \eta).
	\end{align*}
	If $k_j = 0$, then
	\begin{align*}
	\tilde{T}M_{z_j}^{*}( \bm{z}^{\bm{k}} \eta) = \tilde{T}(z_1^{k_1} \cdots z_{j-1}^{k_{j-1}}z_{j+1}^{k_{j+1}} \cdots z_n^{k_n}  M_{z_j}^{*} \eta) & = z_1^{k_1} \cdots z_{j-1}^{k_{j-1}}z_{j+1}^{k_{j+1}} \cdots z_n^{k_n} 
	(T M_{z_j}^{*} \eta)  \\
	& = z_1^{k_1} \cdots z_{j-1}^{k_{j-1}}z_{j+1}^{k_{j+1}} \cdots z_n^{k_n} 
	( M_{z_j}^{*}T \eta) \\
	& =  M_{z_j}^{*}(\bm{z}^{\bm{k}} T\eta) \\
	& = M_{z_j}^{*} \tilde{T}(\bm{z}^{\bm{k}} \eta).
	\end{align*}
	The above fact implies that 
\[	
	 \tilde{T} M_{z_j} = M_{z_j}\tilde{T}  \qquad \mbox{and} \qquad \tilde{T} M_{z_j}^{*}|_{\cll} = M_{z_j}^{*}\tilde{T}
\]	 
   for all $j=1, 2, \cdots, n$. 
	Now the norm preserving operator $\tilde{T}$ on $\cll$ can be extended uniquely by continuity (again denoted by same $\tilde{T}$ )
	to the closure of $\cll$ ( i.e. $\bar{\cll}= H^2_{\tilde{\cle}}(\D^n)$ ) such that
	\[	
	 \tilde{T} M_{z_j} = M_{z_j}\tilde{T}  \qquad \mbox{and} \qquad \tilde{T} M_{z_j}^{*} = M_{z_j}^{*}\tilde{T}
	\]	 
	for all $j=1, 2, \cdots, n$. This 
	imlplies that $\tilde{T}$ is a constant isometric multiplier
	from $H^2_{\tilde{\cle}}(\D^n)$ to $H^2(\D^n)$. Therefore from Lemma \ref{Mul-lemma1}, we have
	$\dim(\tilde{\cle})=1$.  
	Hence $\tilde{T} = M_{\Theta }$, where $\Theta({\bm{z}})  = \Theta(0)$ is
	a unitary from $\tilde{\cle}$ to $\mathbb{C}$ for ${\bm{z}} \in \D^n$.	
	Also
	\[
	T = \tilde{T}|_{\clq}=  M_{\Theta }|_{\clq}.
	\]
	This finishes the proof.
\end{proof}

\begin{thm}
Let $k \in \{1, 2, \ldots, n-1 \}$ be a fixed integer.
Let $\cls = \Theta H^2(\D^k) \otimes \clq_{\theta_1} \otimes \cdots \otimes \clq_{\theta{n-k}}$ 
be a joint $(M_{z_1}, \ldots, M_{z_k}, M_{z_{k+1}}^*, \ldots, M_{z_{n}}^*)$ 
invariant subspace of $ H^2 (\D^n)$, where $\Theta \in H^{\infty}(\D^k)$ is some inner function 
and $\clq_{\theta_j}$ is either a Jordan block or the Hardy space $H^2(\D)$. Let
\begin{align*}
V_j & =P_\cls M_{z_j}  |_{\cls} ~~~ \mbox{for~} 1 \leq j \leq n.
\end{align*}
Then $(V_1, \cdots, V_n)$ is a doubly commuting tuple, where $V_j$ is a pure 
isometry for $1 \leq j \leq k$ and $V_j$ is a pure contraction for $k+1 \leq j \leq n$.
\end{thm}

\begin{proof}
Suppose  $\cls = \Theta H^2(\D^k) \otimes \clq_{\theta_1} \otimes \cdots \otimes \clq_{\theta{n-k}} \subseteq H^2 (\D^n)$ 
is a joint $(M_{z_1}, \ldots, M_{z_k}, M_{z_{k+1}}^*, \ldots, M_{z_{n}}^*)$ 
invariant subspace, where $\Theta \in H^{\infty}(\D^k)$ is some inner function 
and $\clq_{\theta_j}$ is either a Jordan block or the Hardy space $H^2(\D)$. Clearly,
$V_j  =P_\cls M_{z_j}  |_{\cls} = M_{z_j}  |_{\cls}$ for $1 \leq j \leq k$ 
is a pure isometry and $V_j  =P_\cls  M_{z_j}|_{\cls}$ for $ k+1 \leq j \leq n$
is a pure contraction.

Now $\Theta H^2(\D^k)$ is a doubly commuting invariant subspace for $(M_{z_1}, \ldots, M_{z_k})$. 
Therefore $V_iV_j = V_jV_i$ and $V_iV_j^* = V_j^*V_i$ for $1 \leq i < j \leq k$.
Again $\clq_{\theta_1} \otimes \cdots \otimes \clq_{\theta{n-k}}$ is a
doubly commuting invariant subspace for $( M_{z_{k+1}}^*, \ldots, M_{z_{n}}^*)$.
Thus $V_iV_j = V_jV_i$ and $V_iV_j^* = V_j^*V_i$
for $k+1 \leq i< j \leq n$.
Also
\[
V_i V_j^* = M_{z_i}M_{z_j}^*|_{\cls} = M_{z_j}^*M_{z_i}|_{\cls} =  V_j^*V_i
\]
for $1 \leq i \leq k$ and for $k+1 \leq  j \leq n$. 

To prove $V_iV_j =V_j V_i$ for $1 \leq i \leq k$ and for $k+1 \leq  j \leq n$,
we use the standard result of the tensor product of Hilbert spaces: 
Let $\clh_1, \clh_2 \subseteq \clh$ and $h_1, h_2 \in \clh$. 
Then
\[
P_{\clh_1 \otimes \clh_2} (h_1 \otimes h_2) = P_{\clh_1}h_1 \otimes P_{\clh_2}h_2.
\] 

Let $f \in \Theta H^2(\D^k)$ and $g_l \in \clq_{\theta_l}$ for $l=1, \ldots, n-k$. Then 
\begin{align*}
& V_iV_{k+ l}(f \otimes g_1 \otimes \cdots \otimes g_l \otimes \cdots \otimes g_{n-k})\\ 
& = V_i P_{\Theta H^2(\D^k) \otimes \clq_{\theta_1} \otimes \cdots \otimes \clq_{\theta{n-k}}} (f \otimes g_1 \otimes \cdots \otimes M_z g_l\otimes \cdots \otimes g_{n-k})\\
& = V_i(f \otimes P_{\clq_{\theta_1} \otimes \cdots \otimes \clq_{\theta{n-k}}}(g_1 \otimes \cdots \otimes zg_l\otimes \cdots \otimes g_{n-k}))\\
& = M_{z_i}(f \otimes g_1 \otimes \cdots \otimes P_{\clq_{\theta_l}} z g_l\otimes \cdots \otimes g_{n-k}) \\
& = z_i f \otimes g_1 \otimes \cdots \otimes P_{\clq_{\theta_l}} z g_l \otimes \cdots \otimes g_{n-k} \\
& = P_{\Theta H^2(\D^k) \otimes \clq_{\theta_1} \otimes \cdots \otimes \clq_{\theta{n-k}}} (z_i f \otimes g_1 \otimes \cdots \otimes z g_l\otimes \cdots \otimes g_{n-k}) \\
& = V_{k+l} V_i(f \otimes g_1 \otimes \cdots \otimes g_l \otimes \cdots \otimes g_{n-k}).
\end{align*}

This completes the proof.
\end{proof}

We record the following known result on the analytic model 
for a doubly commuting $n$-tuple of operators (see \cite{BNJ}).
Let $T=(T_1, \ldots, T_n)$ be a doubly commuting pure $n$-tuple
of operators on $\clh$. Define the defect operator for the $n$-tuple $T$ as
\[
D_{T^*} = (\prod_{j=1}^n (I_{\clh} - T_j T_j^{*}))^{1 \over 2}.
\]
and 
\[
\cld_{T^*} = \overline{ran}(D_{T^*}).
\]

\begin{thm} \label{Theorem-dc-model}
Let $T=(T_1, \ldots, T_n)$ be a doubly commuting pure $n$-tuple
of operators on $\clh$. Then $(M_{z_1}, \ldots, M_{z_n})$ on
$H^2_{\cld_{T^*}}(\D^n) $ is the minimal isometric dilation of $T$.
That is there exists a joint $(M_{z_1}^*, \ldots, M_{z_n}^*)$
invariant subspace $\clq$ of $H^2_{\cld_{T^*}}(\D^n) $ such that
\[
T_j {\cong} P_{\clq}M_{z_j}\mid_{\clq} \quad (\mbox{for all~} j =1, \ldots, n)
\] 
and 
\[
H^2_{\cld_{T^*}}(\D^n) = \overline{span} \{ \bm{z}^{\bm{k}}\clq : \bm{k} \in \mathbb{Z}_{+}^n \}.
\] 

\end{thm}

Using the above theorem, we characterize a doubly commuting tuple 
of shifts and pure contractions.

\begin{lem}\label{lemma-dc}
Let $k \in \{1, 2, \ldots, n-1 \}$ be a fixed integer.
Let $(T_1, \ldots, T_{k}, T_{{k+1}}, \ldots, T_n)$ be an $n$-tuple of doubly commuting
operators on $\clh$ such that $T_j$ is shift for $1 \leq j \leq k$ and $T_j$ is
a pure contraction for $k+1 \leq j \leq n$. Then $(T_1, \ldots, T_{k}, T_{k+1}, \ldots, T_n)$
on $\clh$ is unitarily eqivalent to 
$(M_{z_1}, \ldots, M_{z_k}, I_{H^2(\D^k)} \otimes \tau_1, \ldots, I_{H^2(\D^k)} \otimes \tau_{n-k}) $ 
on $H^2_{\clw}(\D^k)$, 
where $\clw = \cap_{j=1}^{k}Ker(T_j^{*})$ and $\tau_j$ is a pure contraction on $\clw$
for $1  \leq j \leq n-k$.
\end{lem}

\begin{proof}
Consider the map $\Pi : \clh \rightarrow H^2_{\clw}(\D^k)$ defined by
\[
\Pi (\prod_{j=1}^k T_j^{m_j} \eta) = \prod_{j=1}^k z_j^{m_j} \eta,
\]
where $\eta \in \clw$ and $m_j \in \mathbb{Z}_{+}$. Then clearly $\Pi$
is unitary and $\Pi T_j = M_{z_j}\Pi$ for $1 \leq j \leq k$. Now we know that
if $C \in \clb({H^2_{\clw}(\D^k)})$ such that $M_{z_j} C = C M_{z_j}$ and $M_{z_j}^* C = C M_{z_j}^*$
for $1 \leq j \leq k$, then $C=M_{\tau} = I_{H^2(\D^k)} \otimes \tau$, where $\tau \in \clb(\clw)$. 
Moreover, if $C$ is pure then $\tau$ is pure. Thus
\[
\Pi T_{k+j}\Pi^* = I_{H^2(\D^k)} \otimes \tau_j,
\]
where $\tau_j\in \clb(\clw)$ is a pure contraction for $1 \leq j \leq n-k$.
Since $T_{k+i}T_{k+j} =T_{k+j}T_{k+i}$ and $T_{k+i}^*T_{k+j} =T_{k+j}T_{k+i}^*$,
it readily follows that
\[
\tau_i \tau_j = \tau_j \tau_i  \mbox{~~and ~~} \tau_i^* \tau_j = \tau_j \tau_i^*
\]
for $1 \leq i <j \leq n-k$.
\end{proof}

Now we are in a position to state our main result.

\begin{thm}
Let $\cls$ be a closed subspace of $H^2(\D^n)$ and $k \in \{1, \ldots, n-1 \}$ be a fixed integer. 
Let $\cls$ be a joint $(M_{z_1}, \ldots, M_{z_k}, M_{z_{k+1}}^*, \ldots, M_{z_n}^* ) $ invariant 
subspace and
\begin{align*}
V_j & = P_{{\cls}} M_{z_j} |_{\cls} \quad \mbox{for~~} j=1, \ldots, n.
\end{align*}
If $(V_1, \ldots, V_n)$ is doubly commuting, then 
$\cls= \Theta H^2(\D^k) \otimes \clq_{\theta_1} \otimes \cdots \otimes \clq_{\theta{n-k}}$
for some $\Theta \in H^{\infty}(\D^k)$ inner function 
and $\clq_{\theta_j}$ either a Jordan block or the Hardy space $H^2(\D)$ for 
$j=1, \ldots, n-k$.
\end{thm}

\begin{proof} 
Since $\cls$ is
a joint $(M_{z_1}, \ldots, M_{z_k}, M_{z_{k+1}}^*, \ldots, M_{z_n}^* ) $ invariant subspace
of $H^2(\D^n)$,
\[
V_j = P_{{\cls}} M_{z_j} |_{\cls} =  M_{z_j} |_{\cls}
\]
for $j=1, \ldots, k$ is shift and 
\[
V_j=P_{{\cls}} M_{z_j} |_{\cls}
\]
for $j=k+1, \ldots, n$
is a pure contraction. Hence by the above Lemma \ref{lemma-dc},
there exists a unitary
$\Pi : \cls \rightarrow H^2_{\clw}(\D^k)$ defined by
\[
\Pi (\prod_{j=1}^k V_j^{m_j} \eta) = \prod_{j=1}^k z_j^{m_j} \eta \qquad (\eta \in \clw)
\]
such that
\begin{align*}
\Pi V_j & = M_{z_j}\Pi \quad \mbox{for~} j=1, \ldots, k \\
\Pi V_{k+i} & = (I_{H^2(\D^k)} \otimes \tau_i)\Pi \quad \mbox{for~} i=1, \ldots, n-k 
\end{align*}
where $\clw = \cap_{j=1}^{k}Ker(V_j^{*})$ and $\tau_i$
is a pure contraction on $\clw$ with
\[
\tau_i \tau_j = \tau_j \tau_i  \mbox{~~and ~~} \tau_i^* \tau_j = \tau_j \tau_i^*
\]
for $1 \leq i < j \leq n-k$. Now from Theorem \ref{Theorem-dc-model}, we have  
$\tau^* = (\tau_1^*, \ldots, \tau_{n-k}^* )$ on $\clw$ is unitarily equivalent to
$(M^*_{w_1}|_{\clq}, \ldots, M^*_{w_{n-k}}|_{\clq})$, where $\clq$ is a
joint $(M_{w_1}^*, \ldots,  M_{w_{n-k}}^{*})$ invariant subspace of 
$H^2_{\cld_{\tau^*}}(\D^{n-k})$ and
$\cld_{\tau^*} = \overline{ran}(\prod_{i=1}^{n-k}(I_{\clw} -\tau_i\tau_i^*))^{1 \over 2}$.
Let 
$U: \clw \rightarrow \clq \subseteq H^2_{\cld_{\tau^{*}}}(\D^{n-k})$ be a 
unitary such that
\[
U \tau_i^* = M_{w_i}^*|_{\clq} U  \quad \mbox{for} ~i=1, \ldots, n-k.
\] 
Thus we have a unitary map $I_{H^2(\D^k)} \otimes U^* : H^2(\D^k) \otimes {\clq} \rightarrow H^2(\D^k) \otimes {\clw}$
such that
\[
(I_{H^2(\D^k)} \otimes U^*) (M_{z_j} \otimes I_{\clq})   = (M_{z_j} \otimes I_{\clw}) (I_{H^2(\D^k)} \otimes U^*) 
\]
and
\[
(I_{H^2(\D^k)} \otimes U^*) (I_{H^2(\D^k)} \otimes M_{w_i}^*|_{\clq})  =  
(I_{H^2(\D^k)} \otimes \tau_{i}^*)(I_{H^2(\D^k)} \otimes U^*) 
\]
for $j=1, \ldots, k$ and $i=1, \ldots, n-k$.

Let $\iota: \cls \rightarrow  H^2(\D^k) \otimes H^2(\D^{n-k})$ be the inclusion map.
Therefore we have an operator
\[
\tilde{\Pi} = \iota \circ \Pi^{*} \circ (I_{H^2(\D^k)} \otimes U^*): H^2(\D^k)\otimes {\clq} \rightarrow H^2(\D^k) \otimes H^2(\D^{n-k})
\]
which is an isometry as each one is so and also satisfying
\[
\tilde{\Pi}(M_{z_j} \otimes I_{\clq})  = (M_{z_j} \otimes I_{\clw}) \tilde{\Pi}  \quad \mbox{for~} j=1, \ldots, k
\]
and
\[
\tilde{\Pi}(I_{H^2(\D^k)} \otimes M_{w_i}^*|_{{\clq}})  =  (I_{H^2(\D^k)} \otimes M_{w_i}^*)\tilde{\Pi} \quad \mbox{for~} i=1, \ldots, n-k.
\]
Then from the first equality, we have $\tilde{\Pi}$ is an inner multiplier, say,
\[
\tilde{\Pi} = M_{\Phi},
\]
for some $\Phi \in H^{\infty}_{\clb(\clq, H^2(\D^{n-k}))}(\D^k)$. 
 Again the other equality gives
\[
\Phi(e^{it}) M_{w_j}^*|_{\clq} = M_{w_j}^* \Phi(e^{it}), \quad j=1, \ldots, n-k.
\]
for all $t \in \mathbb{R}^k$ and $e^{it} = e^{it_1} \cdots e^{it_k}$. Since $\Phi(e^{it})$ is isometry almost everywhere (a.e.)
with respect to the Lebesgue measure on the $k$-dimensional torus $\mathbb{T}^k$, from the above Lemma \ref{lemma-extension1}, 
\[
\Phi(e^{it}): \clq \rightarrow H^2(\D^{n-k}) \quad (e^{it} \in \mathbb{T}^k \mbox{~a.e.})
\] 
is the restriction of a constant isometric multiplier
from $H^2_{\tilde{\cle}}(\D^{n-k})$ to $H^2(\D^{n-k})$, where
$\tilde{\cle}$ is a one-dimensional space.
Since $\tilde{\cle}$ and $\mathbb{C}$
are one dimensional spaces, the space of all constant multipliers in
$H^{\infty}_{\clb(\tilde{\cle}, \mathbb{C})}(\D^{n-k})$
is also one dimensional. Hence there exists a unitary multiplier
$M_{{X}}: H_{\tilde{\cle}}^2(\D^{n-k})\rightarrow H^2(\D^{n-k})$
such that
\[
\Phi(e^{it})= \Theta(e^{it}) M_{X}|_{\clq} \quad (\mbox{~a.e. on~} \mathbb{T}^k )
\]  
for some scalar $\Theta(e^{it})$. Since $\Phi$ is analytic,
$\Theta$ is also analytic on $\D^{k}$.
Also $|\Theta(e^{it})|=1$ as $M_{X}|_{\clq}$ and $\Phi(e^{it})$ 
are isometries a.e. on $\mathbb{T}^k$.
Hence $\Theta \in H^{\infty}(\D^k)$ is an inner function.
Thus the inner multiplier $M_{\Phi}: H^2_{\clq}(\D^k) \rightarrow H^2_{H^2(\D^{n-k})}(\D^k)$ factors as 
$M_{\Theta } \otimes M_{X}|_{\clq}$ from $H^2(\D^k) \otimes {\clq}$ to $H^2(\D^k)\otimes {H^2}(\D^{n-k})$.
Therefore
\[
\cls = ran(M_{\Phi}) = M_{\Phi} H^2_{\clq} = \Theta H^2(\D^{k}) \otimes M_{X}(\clq).
\]
Since $M_X: H^2_{\tilde{\cle}}(\D^{n-k}) \rightarrow H^2(\D^{n-k})$ is a constant unitary,
it intertwines with $M_{w_i}^{*}$ for $i=1, \ldots, n-k$.
Hence $M_{X}(\clq)$ is a closed doubly commuting invariant subspace of
$ H^2(\D^{n-k})$. Therefore
from the result \cite{JAY-JORDAN},
$M_{X}(\clq)$ is of the form
\[
M_{X}(\clq) = \clq_{\theta_1} \otimes \cdots \otimes \clq_{\theta{n-k}},
\]
where $\clq_{\theta_j}$ is either a Jordan block or the Hardy space 
$H^2(\D)$, $j=1, \ldots, n-k$.
Thus
\[
\cls = \Theta H^2(\D^{k}) \otimes \clq_{\theta_1} \otimes \cdots \otimes \clq_{\theta{n-k}}.
\]
This finishes the proof.

\end{proof}

\section{Some Applications}

\noindent
In this section we give some concrete examples
of {\it{mixed invariant}} subspaces. 

Let $\mathbb T$ be the unit circle and $K(\cdot,  z)$ be the Szeg\" o kernel on $H^2(\D)$.
We consider the examples 
of {\it{mixed invariant}} subspaces in the Hardy space $H^2(\D^2)$ over the 
bidisc for the sake of simplicity.
We can identify $H^2(\D^2)$ as $H^2(\D) \otimes H^2(\D)$. 
Let $(\alpha_n) \subset \D$ be a sequence such that $\sum_{n=0}^{\infty} 1 - |\alpha_n|^2 < \infty$.
Consider for $N \in \mathbb{N}$,
\[
\clq_N = \mbox{span} \{ K(\cdot, {\alpha_1}), \ldots, K(\cdot, {\alpha_N})\}
\]
and 
\[
\cls_N = H^2(\D) \otimes \clq_N. 
\]
Then clearly $\cls_N$ is invariant under $(M_{z} \otimes I_{H^2(\D)}, I_{H^2(\D)} \otimes M_{z}^{*})$.
Also 
\[
(\cls_N \ominus z_1 \cls_N) = (\cls_N \ominus (M_{z} \otimes I_{H^2(\D)}) \cls_N)=   (H^2(\D) \ominus z H^2(\D)) \otimes \clq_N = \mathbb{C} \otimes \clq_N. 
\]
Thus $\cls_N$ is an example of a \it{mixed invariant} subspace with $\dim(\cls_N \ominus z_1 \cls_N) = N$ 
for each $N \in \mathbb{N}$. Moreover,
if we take 
\[
\clq =  \overline{\mbox{span}} \{ K(\cdot, {\alpha_i}): i \in \mathbb{N} \},
\]
then $\cls = H^2(\D) \otimes \clq$ is a \it{mixed invariant} subspace with 
$\dim(\cls \ominus z_1 \cls) = \infty$. Note that here $\clq$ is a proper 
subspace of $H^2(\D)$ as 
$\prod_{j=1}^{\infty}\frac{z - \alpha_j}{1 - \bar{\alpha_j}z} \in \clq^{\perp}= H^2(\D) \ominus \clq$.

In particular, we have an explicit representation of 
the Beurling-Lax-Halmos's inner function $\Theta$ and the {\it{mixed invariant}} subspace
$\cls$ of $H^2(\D^n)$, where $\cls \ominus z_1 \cls$ is one dimensional.

\begin{thm}\label{thm-3}
Let $\cls \subseteq H^2_{H_n}(\D)$ be
a closed subspace, and let $\clw = \cls \ominus z \cls$. Suppose $\cls$ is invariant 
under $(M_{z}, M_{\kappa_1}^*, \ldots, M_{\kappa_{n}}^* )$. 
Then $\dim(\clw) =1$ if and only if
\[
\cls =  \Theta H^2(\D),
\]
where $\Theta \in H^\infty_{\clb( \mathbb C, H_n) }(\D)$ is an inner multiplier given by
\[
 \Theta(z) = \psi(z) K(\cdot, ( \overline{\phi_1(z)}, \ldots, \overline{\phi_n(z)} ))  \qquad (z \in \D),
\]
$\phi_j, \psi \in H^{\infty}(\D)$ and $|\phi_j(z)|< 1$ for all $z \in \D$ and 
\[
{|\psi(e^{i\theta})|^2} = {\prod_{j=1}^n(1- |\phi_j(e^{i\theta})|^2)}  \mbox{~~a.e. on~~} \mathbb T.
\]
\end{thm}

\begin{proof}
Suppose that $\cls$ is a closed subspace of $H^2_{H_n}(\D)$ which 
is invariant under $(M_{z}, M_{\kappa_1}^*, \ldots, M_{\kappa_{n}}^* )$.
Consider 
\begin{align*}
V & = M_z|_{\cls} ~~~ \quad \mbox{and} \\ 
V_j & = P_{\cls}M_{\kappa_j}|_{\cls} \mbox{~~for~}j =1, \ldots, n. 
\end{align*}
Since $V$ is a shift and $V_j^*$ commute with $V$, the tuple
$(V, V_1^{*}, \ldots, V_n^*)$ on $\cls$ is unitarily equivalent to $(M_z, M_{\phi_1}, \ldots, M_{\phi_n})$
on $H^2_{\clw}(\D)$, where $\clw = \cls \ominus z\cls$ and from Remark \ref{Remark-commutant-1},
\[
\phi_j(w) = P_{\clw}(I_\cls - w V^{*})^{-1}V_j^{*}|_{\clw} \quad (w \in \D, ~j =1, \ldots, n).
\] 
Since $\cls$ is joint invariant under $(M_{z}, M_{\kappa_1}^*, \ldots, M_{\kappa_{n}}^* )$
and $\dim(\clw) =1$, using the result (\cite{AASS}, Theorem 3.2 ), 
we have an isometry
\[
M_{\Theta}: H^2 (\D) \rightarrow H^2_{H_n}(\D) \mbox{~~such that~~} \cls = \Theta H^2{(\D)},
\]
where $\Theta \in H^\infty_{\clb( \mathbb C, H_n) }(\D)$ is an inner multiplier
and $M^{*}_{\kappa_j}M_{\Theta} = M_{\Theta} M_{\phi_j}$ for some $\phi_j \in H^{\infty}(\D)$
for $1\leq j \leq n$.
Therefore, for each $z \in \D$
\[
M^{*}_{z_j}\Theta(z) = \Theta(z)\phi_j(z)= \phi_j(z)\Theta(z)
\]
for $1\leq j \leq n$.
This implies $\Theta(z) \cdot 1 = \Theta(z) \in H^2(\D^n)$ is an eigenvector for the operator 
$M_{z_j}^{*}$
with eigenvalue $\phi_j(z)$. Hence 
$|\phi_j(z)|< 1$ for $z \in \D$, $1\leq j \leq n$.
That means for each $z \in \D$
\[
\Theta(z) \in \ker( M_{z_j} - \overline{\phi_j(z)} I_{H^2(\D^n)})^{*}  \quad \mbox{for all~} 1\leq j \leq n.
\]
This implies
$ \Theta(z) \in \mbox{span} \displaystyle\{K(\cdot, ( \overline{\phi_1(z)}, \ldots, \overline{\phi_n(z)} )\}$,
where
\[
K(\cdot, ( \overline{\phi_1(z)}, \ldots, \overline{\phi_n(z)} )(\bm{w}) =  \prod_{j=1}^n \frac{1}{1- \phi_j(z)w_j} \qquad (\bm{w} \in \D^n).
\]
Therefore
\[
\Theta(z) = \psi(z) K(\cdot, ( \overline{\phi_1(z)}, \ldots, \overline{\phi_n(z)} ) \mbox{~~for some~~} \psi.
\]
Using the fact that
\[
\| K(\cdot, ( \overline{\phi_1(z)}, \ldots, \overline{\phi_n(z)} ) \|^2 =  {1 \over { \prod_{j=1}^n(1- |\phi_j(z)|^2)}},
\]
we have from the above identity
\[
|\psi(z)|^2 = \|\Theta(z)\|^2 { \prod_{j=1}^n(1- |\phi_j(z)|^2)}.
\]
It is easy to see that $\psi \in H^{\infty}(\D)$ as 
$|\psi(z)| \leq \|\Theta(z) \| \leq 1$ for $z \in \D$.
Since $\Theta \in H^\infty_{\clb( \mathbb C, H_n )}(\D)$ is an inner multiplier, by taking the 
radial limit, we have 
\[
{|\psi(e^{i\theta})|^2} = {\prod_{j=1}^n(1- |\phi_j(e^{i\theta})|^2)}  \mbox{~~a.e. on~~} \mathbb T.
\]
 Also
\[
\cls = \Theta H^2(\D) 
\]
where
\[
 \Theta(z) = \psi(z) K(\cdot, ( \overline{\phi_1(z)}, \ldots, \overline{\phi_n(z)} ))  \qquad (z \in \D).
\]
The converse follows easily.

This completes the proof.
\end{proof}

The above result is a generalization of the result on the Hardy space $H^2(\D^2)$ over the bidisc  
by Izuchi et al.(see Theorem 3.2 ,\cite{IIN}).

\vspace{0.3in}

\NI\textit{Acknowledgement:} 
The authors are grateful to the anonymous reviewer for his/her critical 
and constructive reviews and suggestions that have
substantially improved the presentation of the paper.
The authors are also thankful to Prof. Jaydeb Sarkar for many fruitful 
discussions and his valuable comments. 
The first author's research work is
supported by Faculty Initiation Grant (FIG scheme), 
IIT Roorkee (Ref. No: MAT/FIG/100820) and he also 
acknowledges Indian Statistical Institute,
Bangalore Centre for warm hospitality.

\end{document}